\documentclass[final, 3p, times]{elsarticle}%
\usepackage{graphicx}
\usepackage{xcolor}
\usepackage{colortbl}
\usepackage{epstopdf}
\usepackage{amssymb}
\usepackage{amsthm}
\usepackage{amsmath}
\usepackage{color}
\usepackage{dsfont}
\usepackage{mathtools}
\usepackage{extarrows}
\usepackage{hyperref}
\usepackage{bm}
\usepackage{caption}
\usepackage{float}
\usepackage{verbatim}
\usepackage{epsfig}
\usepackage{bm}
\usepackage{multirow}
\usepackage[section]{placeins}
\usepackage{stmaryrd}
\usepackage{subfigure,tabularx,pdfpages}

\numberwithin{equation}{section}
\biboptions{compress}

\usepackage{amsmath}
\usepackage{amssymb}
\usepackage{amsthm,mathtools}
\usepackage[mathscr]{eucal}
\usepackage{xcolor}
\usepackage{fontenc}
\usepackage{graphicx}
\usepackage{geometry}
\theoremstyle{definition}
\newtheorem{definition}{Definition}[section]
\newtheorem{theorem}[definition]{Theorem}
\newtheorem*{theorem*}{Conjecture}
\newtheorem{proposition}[definition]{Proposition}
\newtheorem{lemma}[definition]{Lemma}

\theoremstyle{remark}
\newtheorem{remark}[definition]{Remark}
\newtheorem{example}[definition]{Example}
\newcounter{enumctr}

\newcommand{\R}{\mathbb{R}}
\newcommand{\N}{\mathbb{N}}

\begin{document}
\begin{frontmatter}
\title{Fractional coupled Halanay inequality and its applications}

\tnotetext[label1]{The research of Dongling Wang is supported in part by the National Natural Science Foundation of China under grants 12271463.
\\ Declarations of interest: none.
}

\author[VN1]{La Van Thinh} 
\ead{ lavanthinh@hvtc.edu.vn;}
\author[VN2]{Hoang The Tuan} 
\ead{httuan@math.ac.vn;}
\author[XTU]{Dongling Wang \corref{mycorrespondingauthor}}
\ead{wdymath@xtu.edu.cn; ORCID 0000-0001-8509-2837}
\cortext[mycorrespondingauthor]{Corresponding author. }
\author[XTU]{Yin Yang} 
\ead{yangyin2020@xtu.edu.cn}

\address[VN1]{Academy of Finance, Le Van Hien Street, Duc Thang Ward, Bac Tu Liem District, Hanoi, Viet Nam}
\address[VN2]{Institute of Mathematics, Vietnam Academy of Science and Technology, 18 Hoang Quoc Viet, 10307 Ha Noi, Viet Nam}
\address[XTU]{Hunan Key Laboratory for Computation and Simulation in Science and Engineering,\\ School of Mathematics and Computational Science, Xiangtan University, Xiangtan, Hunan 411105, China}

\begin{abstract}
This paper introduces a generalized fractional Halanay-type coupled inequality, which serves as a robust tool for characterizing the asymptotic stability of diverse time fractional functional differential equations, particularly those exhibiting Mittag-Leffler type stability. 
Our main tool is a sub-additive property of Mittag-Leffler function and its optimal asymptotic decay rate estimation. Our results further optimize and improve some existing results in the literature. We illustrate two significant applications of this fractional Halanay-type inequality.
Firstly, by combining our results in this work with the positive representation method
positive representation of delay differential systems, we establish an asymptotic stability criterion for a category of linear fractional coupled systems with bounded delays. This criterion extends beyond the traditional boundaries of positive system theory, offering a new perspective on stability analysis in this domain. Secondly, through energy estimation, we establish the contractility and dissipativity of a class of time fractional neutral functional differential equations. Our analysis reveals the typical long-term polynomial decay behavior inherent in time fractional evolutionary equations, thereby providing a solid theoretical foundation for subsequent numerical investigations. 
\end{abstract}
	
\end{frontmatter}

    \section{Introduction}
   Time fractional order models and the corresponding Volterra integral equations with singular kernels have seen remarkable attention and development in recent years. The primary reason for this growing interest lies in their ability to more accurately capture various forms of anomalous diffusion phenomena, such as sub-diffusion and super-diffusion, when compared to traditional classical integer-order solution equations \cite{Kai, Met20, Jin21}. 
 A natural explanation for fractional calculus in mathematics or physics can be found in the random walk model. This model demonstrates that the mean square displacement of particles in anomalous diffusion processes follows a power-law relationship with time as $\langle x(t) \rangle \sim C t^{\alpha}$, where the parameter $\alpha\in(0,1) \cap (1, 2)$. In contrast, the standard diffusion process exhibits a linear relationship between the mean square displacement and time, expressed as $ \langle x(t) \rangle \sim C t$ \cite{Met20}.
   
The non-local nature of time fractional derivatives and weak singular kernels result in polynomial decay of solutions in anomalous diffusion models over time, also known as Mittag-Leffler stability \cite{Zacher15}. This is completely different from the exponential decay rate of the standard diffusion model solution over time.
On the other hand, the phenomenon of delay effects is prevalent in various practical models. These delay phenomena primarily arise from the time required for the transmission of materials, energy, and information between different components of a system. Specifically, these delays including physical delays, biological delays, and informational delays. When we consider both these delay effects and the memory characteristics of a system during modeling, we can derive various time-fractional delay differential equations. These equations not only provide a more accurate description of the system's dynamic behavior but also help us better understand and predict the system's response across different time scales. For example, in biomedical science, considering the delay effects in disease transmission can help us more accurately predict the progression of an epidemic. In engineering control, taking into account the delay characteristics of a system can enhance the stability and efficiency of control algorithms \cite{WHO, WangWS}.

The long-term asymptotic behavior of the solution is an important qualitative characteristic of an equation. For delay equations, or more generally functional differential equations, a key tool for studying the long-term asymptotic stability of solutions is the Halanay-type inequality \cite{WangWS}. This inequality characterizes the asymptotic stability and exponential decay rate of solutions under various delay effects.

For time fractional functional differential equations, a fractional version of the Halanay inequality was first introduced in \cite{Wang_15}, proving the asymptotic stability, but did not characterize the decay rate of the solution.  By essentially improving the results in \cite{Wang_15}, an improved fractional Halanay inequality was established in \cite{Wang_19}, which exhibits the polynomial decay rate of solutions to fractional  functional differential equations. Thus distinguishing fractional functional differential equations from integer ones in terms of long-term qualitative behavior. Some other related studies, see for example \cite{HKT_23, NNThang, KT23, Wang, ChengHu}.

We briefly review the main ideas of \cite{Wang_15, Wang_19} along with their generalizations.
In \cite{Wang_15}, the authors use the supremum theorem and mathematical induction to achieve asymptotic stability. However, this approach fails to distinguish decay rates between fractional and integer systems.
In \cite{Wang_19}, the authors assume a Mittag-Leffler function representation for the solution and derive the fractional version characteristic equation. By analyzing its strict negative roots, they obtain the optimal algebraic decay rate of the solution, known as Mittag-Leffler stability.

However, due to the Mittag-Leffler functions, the fractional characteristic equation becomes very complex, 
making it more difficult to analyze the distribution of its roots than the corresponding integer order models.
In fact, for the fractional characteristic functions introduced in \cite{Wang_19}, strict analysis of their root negativity requires great caution, 
as pointed out in \cite{KT23, NNThang}.
A powerful tool for simplifying the analysis of fractional characteristic equations is the sub-linear of Mittag-Leffler functions, as presented in Lemma \ref{Est-ML} below. Our main contributions in this article are twofold: 

\begin{itemize}
\item
Taking note of the sub-linear of the Mittag-Leffler functions, we introduce a new characteristic equation that is much simpler compared to the previous one, making it clear and concise to analyze the strict negativity of its roots. This leads to Mittag-Leffler stability estimation of the solution.

\item 
The coupled fractional Halanay inequality we consider in this article is very general, and many of the relevant results in the literature can be regarded as special cases of it.  As a typical application, we establish the dissipativity and contractivity of solutions to fractional neutral functional differential equations.
\end{itemize}

We close this section by introducing some concepts and definitions that will be used throughout the article. Let $\N, \R, \, \R_{\ge0}$ be the set of natural numbers, real numbers, and nonnegative real numbers, respectively. Let  $d\in\N$ and $\R^d$ stands for the $d$-dimensional real Euclidean space with the standard inner product $\langle\cdot,\cdot\rangle$ and norm $\|\cdot\|$. Denote by $\R^d_{\geq 0}$ the set of all vectors in $\R^d$ with nonnegative entries, that is, $\R_{\geq 0}^d=\left\{x=(x_1,\dots,x_d)^{\rm T}\in \R^d:x_i\ge0,\ 1\le i\le d\right\}$.
For any $m,n\in \N$, let $\R^{m\times n}$ denote the set of matrices with  with $m$ rows and $n$ columns in $\R$. Given a closed interval $J\subset \R$ and $X$ is a subset of $\R^d$, we define $C(J; X)$ as the set of all continuous functions from $J$ to $X$.

    For $\alpha \in (0,1]$ and $T>0$,  the Riemann--Liouville fractional integral of a function $x:[0,T] \rightarrow \mathbb R$ is defined by 
\[
I^\alpha_{0^+}x(t) := \frac{1}{\Gamma(\alpha)}\int_{0}^{t}(t-s)^{\alpha -1}x(s)ds,\;t\in (0,T], 
\]
	and its Caputo fractional derivative of the order $\alpha$ as $ {D^\alpha_{0^+}}x(t) := \frac{d}{dt}I^{1-\alpha}_{0^+}(x(t)-x(0)), t\in (0,T]$, 
	where $\Gamma(\cdot)$ is the Gamma function and $\frac{d}{dt}$ is the usual derivative (see, e.g., 
	\cite[Chapters 2 and 3]{Kai} and \cite{Vainikko_16} for more detail on fractional calculus). For $d\in\N$ and a vector-valued function $x(\cdot)$ in $\R^d,$ we use the notation ${D^{\alpha}_{0^+}}x(t):=({D^{\alpha}_{0^+}}x_1(t),\dots,{D^{\alpha}_{0^+}}x_d(t))^{\rm T}.$

\section{Generalized Fractional Halanay Inequality}

\subsection{Preliminary}
We provide a summary of the key preparation results needed for the subsequent analysis presented throughout the paper.
Let $\alpha,\beta>0$. The Mittag-Leffler function $E_{\alpha,\beta}(\cdot):\R\rightarrow \R$ is defined by
\[
E_{\alpha,\beta}(t)=\sum_{k=0}^\infty \frac{t^k}{\Gamma(\alpha k+\beta)},\;\forall t\in \R.
\]
In the case $\beta=1$, for simplicity we use convention $E_\alpha(t):=E_{\alpha,1}(t)$ for all $t\in\R$.

\begin{lemma} \cite{Gorenflo_B} \label{ML1}
	The following statements hold. 
	\begin{itemize}
		\item[(i)] For $\alpha>0$ and $t\in\R$, then $E_{\alpha}(t)>0,\ E_{\alpha,\alpha}(t)>0$.
		\item[(ii)] For $\alpha>0$ and $t\in \R$, then $\displaystyle\frac{d}{dt}E_{\alpha}(t)=\frac{1}{\alpha}E_{\alpha,\alpha}(t)$.
  \item[(iii)] Let $\alpha\in (0,1)$, for all $t>0$, then $^{\!C}D^{\alpha}_{0^+}E_{\alpha}(\lambda t^\alpha)=\lambda E_{\alpha}(\lambda t^\alpha),\;\forall \lambda\in \R$.
		\item[(iv)] Let $\alpha>0$, then $E_{\alpha}(-t)>E_{\alpha}(-s)$ for all $0\leq t<s.$ 
	\end{itemize}
\end{lemma}

The following result offers a useful estimate for classical Mittag-Leffler functions, known as their sub-additive property.

\begin{lemma} \cite[Lemma 3.2]{NNThang} \label{Est-ML}
For $\lambda>0$ and $t,s\ge0$, we have
\[E_{\alpha}(-\lambda t^{\alpha})E_{\alpha}(-\lambda s^{\alpha})\le E_{\alpha}(-\lambda (t+s)^{\alpha}).\]
In addition, if $\lambda>0$, $t,s\geq 0$ and $ts>0$, then
\[E_{\alpha}(-\lambda t^{\alpha})E_{\alpha}(-\lambda s^{\alpha})< E_{\alpha}(-\lambda (t+s)^{\alpha}).\]
\end{lemma}

\begin{lemma} \cite[Lemma 25]{Cong-Tuan-Hieu} \label{compare-FDE}
Let $x : [0, T] \rightarrow \R$ be continuous and the Caputo fractional derivative
$^{\!C}D^{\alpha}_{0^+}x(t)$ exists on the interval $(0, T]$. If there exists $t_1\in (0,T]$ such that $x(t_1)=0$ and $x(t)<0,\ \forall t\in [0,t_1)$, then $$^{\!C}D^{\alpha}_{0^+}x(t_1)\geq 0.$$
\end{lemma}

\subsection{Generalized Fractional Halanay Inequality}

This section is devoted to prove the main result of the paper. In particular, we develop a generalized Halanay-type coupled inequality within the framework of fractional calculus.
\begin{theorem} \label{Theorem-Hala}
	Let $\tau>0$ and $u,v:[-\tau,+\infty) \rightarrow \R_{\geq 0}$ be continuous functions. Suppose that $a,b,c,d,e,f$ are non-negative parameters such that $0\le d<1,b\ge 0,c\ge 0,a>b+\displaystyle\frac{ec}{1-d}>0.$ Assume that the Caputo fractional derivative
	$^{\!C}D^{\alpha}_{0^+}u(t)$ exists on the interval $(0, \infty)$ and 
	\begin{align} 
			^{\!C}D^{{\alpha}}_{0+}u(t)&\le -au(t)+b\displaystyle\sup_{0\le s_1\le \tau_1}u(t-s_1)+e\sup_{0\le s_2\le \tau_2}v(t-s_2)+f,\; t\in (0,\infty),\label{E0_1}\\
			v(t)&\le \hspace{0.3cm}cu(t)+d\displaystyle\sup_{0\le s_3\le \tau_3}v(t-s_3),\; t\in [0,\infty),\label{E0_2}\\
			u(\cdot)&=\varphi(\cdot),\;v(\cdot)=\psi(\cdot)\;\text{on}\; [-\tau,0],\label{E0_3}
	\end{align}
where $\varphi,\ \psi:[-\tau,0] \rightarrow \R_{\geq 0}$ are given continuous functions 
and $\tau_1,\tau_2,\tau_3$ are positive delays and $\tau=\max\left\{\tau_1,\tau_2,\tau_3\right\}$. Then, there are positive constants $M>0, \lambda^*>0$ satisfying
\begin{align} 
	u(t)&\leq ME_{\alpha}(-\lambda^*t^\alpha)+\frac{(1-d)f}{(1-d)(a-b)-ce},\label{est_hala_1}\\
v(t)&\leq \frac{Mc}{E_{\alpha}(-\lambda^*\tau_3^\alpha)-d}E_{\alpha}(-\lambda^*t^\alpha)+\frac{cf}{(1-d)(a-b)-ce}\label{est_hala_2}
\end{align}
for all $t\ge 0$, where $\lambda^*$ is a unique positive solution of the following equation 
\begin{align} \label{eq:hfun}
h(\lambda):=\lambda-a+\frac{b}{E_{\alpha}(-\lambda \tau_1^\alpha)}+\frac{ec}{E_{\alpha}(-\lambda \tau_3^\alpha)-d}\times\frac{1}{E_{\alpha}(-\lambda \tau_2^\alpha)}=0.
\end{align}
\end{theorem}

\begin{proof}
	Define $h(\lambda):=\displaystyle\lambda-a+\frac{b}{E_{\alpha}(-\lambda \tau_1^\alpha)}+\frac{ec}{E_{\alpha}(-\lambda \tau_3^\alpha)-d}\times\frac{1}{E_{\alpha}(-\lambda \tau_2^\alpha)}$.
    By the properties of Mittag-Leffler function, it is easy to see there exists a unique positive constant $\lambda_0>0$ such that $ E_{\alpha}(-\lambda_0 \tau_3^\alpha)=d.$
	Notice that
	$$h(0)=-a+b+\frac{ec}{1-d}<0,\ \displaystyle\lim_{\lambda\to {\lambda_0}_{-}}h(\lambda)=+\infty.$$
	Moreover, $h(\cdot)$ is continuously differentiable on $(0,\lambda_0)$ and
	\begin{align*}
	h'(\lambda)=&1+\frac{b\tau_1^\alpha E_{\alpha,\alpha}(-\lambda \tau_1^\alpha)}{\alpha E_{\alpha}(-\lambda \tau_1^\alpha)^2}+\frac{\tau_3^\alpha E_{\alpha,\alpha}(-\lambda \tau_3^\alpha)}{\alpha\left(E_{\alpha}(-\lambda \tau_3^\alpha)-d\right)^2}\times\frac{ec}{E_{\alpha}(-\lambda \tau_2^\alpha)}\\
	&+\frac{ec}{E_{\alpha}(-\lambda \tau_3^\alpha)-d}\times\frac{\tau_2^\alpha E_{\alpha,\alpha}(-\lambda \tau_2^\alpha)}{\alpha E_{\alpha}(-\lambda \tau_2^\alpha)^2}>0	
	\end{align*}
	for all $\lambda\in (0,\lambda_0)$. Thus, $h(\cdot)$ is strictly increasing on $[0,\lambda_0)$ and the equation $h(\lambda)=0$ has a unique positive solution $\lambda=\lambda^*\in (0,\lambda_0)$. 
    
	Because $\varphi,\psi$ are continuous on $[-r,0]$, taking $M>0$ such that
	\begin{align}
 \sup_{0\le s\le \tau}\varphi(s-\tau)\leq M, \quad
 \sup_{0\le s\le \tau}\psi(s-\tau)\leq \frac{Mc}{E_{\alpha}(-\lambda^* \tau_3^\alpha)-d}.\label{icd_2}
 \end{align}
 Let $\varepsilon>0$ be small enough such that $\varepsilon\in (0,\lambda^*)$. We first show that for all $t\geq 0$
 \begin{align} 
	u(t)&< (M+\varepsilon)E_{\alpha}(-(\lambda^*-\varepsilon)t^\alpha)+\gamma_1, \label{est_hala_add_1}\\ 
 v(t)&< \frac{(M+\varepsilon)c}{E_{\alpha}(-(\lambda^*-\varepsilon)\tau_3^\alpha)-d}E_{\alpha}(-(\lambda^*-\varepsilon)t^\alpha)+\gamma_2,\label{est_hala_add_2}
\end{align}
where $\gamma_1,\gamma_2$ are non-negative parameters chosen later.
	We do this by contradiction. Due to $u(0)=\varphi(0)< M+\varepsilon$ and $v(0)=\psi(0)< \frac{(M+\varepsilon)c}{E_{\alpha}(-\lambda^* \tau_3^\alpha)-d}$, thus if the statements \eqref{est_hala_add_1}--\eqref{est_hala_add_2} are not true then there exists $t_*>0$ such that one of the following assertions holds.
	\begin{itemize}
		\item[(i)] \begin{align}
  u(t_*)&=(M+\varepsilon)E_{\alpha}(-(\lambda^*-\varepsilon)t^\alpha_*)+\gamma_1,\label{contra_1}\\
  u(t)&<(M+\varepsilon)E_{\alpha}(-(\lambda^*-\varepsilon)t^\alpha)+\gamma_1,\ \forall t\in [0,t_*),\label{contra_2}
  \end{align} and
		\begin{equation}\label{contra_3}
  v(t)\le \frac{(M+\varepsilon)c}{E_{\alpha}(-(\lambda^*-\varepsilon)\tau_3^\alpha)-d}E_{\alpha}(-(\lambda^*-\varepsilon)t^\alpha)+\gamma_2,\  \forall t\in[0,t_*].
  \end{equation}
		\item[(ii)] \begin{align}
		    v(t_*)&= \displaystyle\frac{(M+\varepsilon)c}{E_{\alpha}(-(\lambda^*-\varepsilon)\tau_3^\alpha)-d}E_{\alpha}(-(\lambda^*-\varepsilon)t_*^\alpha)+\gamma_2,\label{contra_4}\\
       v(t)&< \frac{(M+\varepsilon)c}{E_{\alpha}(-(\lambda^*-\varepsilon)\tau_3^\alpha)-d}E_{\alpha}(-(\lambda^*-\varepsilon)t^\alpha)+\gamma_2,\; \forall t\in[0,t_*)\label{contra_5}
		\end{align} and
		\begin{equation}\label{contra_6}
  u(t)\le (M+\varepsilon)E_{\alpha}(-(\lambda^*-\varepsilon)t^\alpha)+\gamma_1,\ \forall t\in [0,t_*].\end{equation}
	\end{itemize}
	
 Suppose the counterfactual $\textup{(i)}$ holds. Let $\gamma_1,\gamma_2\geq 0$ satisfies
 $-a\gamma_1+b\gamma_1+c\gamma_2+f=0$.
 Set $z(t)=u(t)-(M+\varepsilon)E_{\alpha}(-(\lambda^*-\varepsilon)t^\alpha)-\gamma_1,\ t\ge0.$
	Then, $z(t_*)=0 \text{ and } z(t)<0,\ \forall t\in [0,t_*).$
	By Lemma \ref{compare-FDE}, 
	\begin{equation}\label{contra1} 
	^{\!C}D^{\alpha}_{0^+}z(t_*)\geq 0.
 \end{equation} 
 
 On the other hand, 
	\begin{align*}
    ^{\!C}D^{\alpha}_{0^+}z(t_*)&=^{\!C}D^{\alpha}_{0^+}u(t_*)+(M+\varepsilon)(\lambda^*-\varepsilon)E_{\alpha}(-(\lambda^*-\varepsilon)t_*^\alpha) \\
		&\le -au(t_*)+b\sup_{0\le s_1\le\tau_1}u(t_*-s_1)+e\sup_{0\le s_2\le\tau_2}v(t_*-s_2)+f+(M+\varepsilon)(\lambda^*-\varepsilon)E_{\alpha}(-(\lambda^*-\varepsilon)t_*^\alpha)\\
		&=(M+\varepsilon)(\lambda^*-\varepsilon)E_{\alpha}(-(\lambda^*-\varepsilon)t_*^\alpha)-a(M+\varepsilon)E_{\alpha}(-(\lambda^*-\varepsilon)t_*^\alpha)-a\gamma_1+b\sup_{0\le s_1\le\tau_1}u(t_*-s_1)\\
  &+e\sup_{0\le s_2\le\tau_2}v(t_*-s_2)+f.
	\end{align*}
\textbf{Case I.1:} $t_*>\max\left\{\tau_1,\tau_2\right\}$. In this case, it is obvious to see that
\begin{equation*}
	\begin{cases}
		\displaystyle\sup_{0\le s_1\le\tau_1}u(t_*-s_1)&\leq (M+\varepsilon)E_{\alpha}(-(\lambda^*-\varepsilon)(t_*-\tau_1)^\alpha)+\gamma_1,\\
		\displaystyle\sup_{0\le s_2\le\tau_2}v(t_*-s_2)&\leq \displaystyle\frac{(M+\varepsilon)c}{E_{\alpha}(-(\lambda^*-\varepsilon)\tau_3^\alpha)-d}E_{\alpha}(-(\lambda^*-\varepsilon)(t_*-\tau_2)^\alpha)+\gamma_2.
	\end{cases}
\end{equation*}
 From this, using Lemma \ref{Est-ML}, it implies
\begin{align}
	^{\!C}D^{\alpha}_{0^+}z(t_*)&\le (M+\varepsilon)(\lambda^*-\varepsilon)E_{\alpha}(-(\lambda^*-\varepsilon)t_*^\alpha)-a(M+\varepsilon)E_{\alpha}(-(\lambda^*-\varepsilon)t_*^\alpha)-a\gamma_1\notag\\
    &\hspace{1cm}+b(M+\varepsilon) E_{\alpha}(-(\lambda^*-\varepsilon)(t_*-\tau_1)^\alpha)+b\gamma_1+c\gamma_2+f\notag\\
    &\hspace{1cm}+\frac{(M+\varepsilon)ec}{E_{\alpha}(-(\lambda^*-\varepsilon)\tau_3^\alpha)-d}E_{\alpha}(-(\lambda^*-\varepsilon)(t_*-\tau_2)^\alpha)\notag\\
	&\le (M+\varepsilon)\Big[(\lambda^*-\varepsilon)E_{\alpha}(-(\lambda^*-\varepsilon)t_*^\alpha)-aE_{\alpha}(-(\lambda^*-\varepsilon)t_*^\alpha)+\frac{bE_{\alpha}(-(\lambda^*-\varepsilon)t_*^\alpha)}{E_{\alpha}(-(\lambda^*-\varepsilon)\tau_1^\alpha)}\Big]\notag \\
	&\hspace{0.5cm}+\frac{(M+\varepsilon)ec}{E_{\alpha}(-(\lambda^*-\varepsilon)\tau_3^\alpha)-d}\times\frac{E_{\alpha}(-(\lambda^*-\varepsilon)t_*^\alpha)}{E_{\alpha}(-(\lambda^*-\varepsilon)\tau_2^\alpha)}-a\gamma_1+b\gamma_1+c\gamma_2+f.\notag
    \end{align}
    This means that
    \begin{align}
    \frac{^{\!C}D^{\alpha}_{0^+}z(t_*)}{(M+\varepsilon)E_{\alpha}(-(\lambda^*-\varepsilon)t_*^\alpha)}
	&=\lambda^*-\varepsilon-a+\frac{b}{E_{\alpha}(-(\lambda^*-\varepsilon)\tau_1^\alpha)}\notag\\
    &\hspace{-2cm}+\frac{ec}{E_{\alpha}(-(\lambda^*-\varepsilon)\tau_3^\alpha)-d}\times\frac{1}{E_{\alpha}(-(\lambda^*-\varepsilon)\tau_2^\alpha)}+\frac{-a\gamma_1+b\gamma_1+c\gamma_2+f}{(M+\varepsilon)E_{\alpha}(-(\lambda^*-\varepsilon)t_*^\alpha)}\notag\\
	&=h(\lambda^*-\varepsilon)\;\;\notag\\
 &<h(\lambda^*) =0.\notag
\end{align}
The last estimate above is due to the fact that $h(\cdot)$ is strictly increasing on $(0,\lambda_0)$, $h(\lambda^*)=0$ and thus it shows a contradiction to \eqref{contra1}.

\textbf{Case I.2:} $t_*<\min\left\{\tau_1,\tau_2\right\}$. In this case, we obtain
\begin{equation*}
	\begin{cases}
		\displaystyle\sup_{0\le s_1\le\tau_1}u(t_*-s_1)&\leq M+\varepsilon+\gamma_1,\\
		\displaystyle\sup_{0\le s_2\le\tau_2}v(t_*-s_2)&\leq \displaystyle\frac{(M+\varepsilon)c}{E_{\alpha}(-(\lambda^*-\varepsilon)\tau_3^\alpha)-d}+\gamma_2.
	\end{cases}
\end{equation*}
Thus, utilizing Lemma \ref{ML1}, it deduces
\begin{align}
	^{\!C}D^{\alpha}_{0^+}z(t_*)&\le (M+\varepsilon)(\lambda^*-\varepsilon)E_{\alpha}(-(\lambda^*-\varepsilon)t_*^\alpha)-a(M+\varepsilon)E_{\alpha}(-(\lambda^*-\varepsilon)t_*^\alpha)+b(M+\varepsilon)\notag\\
    &\hspace{1cm}+\frac{(M+\varepsilon)ec}{E_{\alpha}(-(\lambda^*-\varepsilon)\tau_3^\alpha)-d}-a\gamma_1+b\gamma_1+c\gamma_2+f.\notag
    \end{align}
    From this, we have
    \begin{align}
    \frac{^{\!C}D^{\alpha}_{0^+}z(t_*)}{(M+\varepsilon)E_{\alpha}(-(\lambda^*-\varepsilon)t_*^\alpha)}&=\lambda^*-\varepsilon-a+\frac{b}{E_{\alpha}(-(\lambda^*-\varepsilon)t_*^\alpha)}
   +\frac{ec}{E_{\alpha}(-(\lambda^*-\varepsilon)\tau_3^\alpha)-d}\times\frac{1}{E_{\alpha}(-(\lambda^*-\varepsilon)t_*^\alpha)}\notag\\
	&\le \lambda^*-\varepsilon-a+\frac{b}{E_{\alpha}(-(\lambda^*-\varepsilon)\tau_1^\alpha)}+\frac{ec}{E_{\alpha}(-(\lambda^*-\varepsilon)\tau_3^\alpha)-d}\times\frac{1}{E_{\alpha}(-(\lambda^*-\varepsilon)\tau_2^\alpha)}\notag\\
	&=h(\lambda^*-\varepsilon)\notag\\
 &<h(\lambda^*)=0, \notag
\end{align}
a contradiction to \eqref{contra1}.

\textbf{Case I.3:} $\min\left\{\tau_1,\tau_2\right\}\le t_*\le\max\left\{\tau_1,\tau_2\right\}$. There are the following two possibilities. 

\textbf{Case I.3.1:} $\min\left\{\tau_1,\tau_2\right\}\le t_*\le\max\left\{\tau_1,\tau_2\right\}$, $\tau_1\le\tau_2$. Then, 	
	\begin{equation*}
		\begin{cases}
			\displaystyle\sup_{0\le s_1\le\tau_1}u(t_*-s_1)&\leq (M+\varepsilon)E_{\alpha}(-(\lambda^*-\varepsilon)(t_*-\tau_1)^\alpha)+\gamma_1,\\
			\displaystyle\sup_{0\le s_2\le\tau_2}v(t_*-s_2)&\leq \displaystyle\frac{(M+\varepsilon)c}{E_{\alpha}(-(\lambda^*-\varepsilon)\tau_3^\alpha)-d}+\gamma_2.
		\end{cases}
	\end{equation*}
By combining Lemma \ref{ML1} and Lemma \ref{Est-ML}, it shows that
\begin{align}
	\frac{^{\!C}D^{\alpha}_{0^+}z(t_*)}{(M+\varepsilon)E_{\alpha}(-(\lambda^*-\varepsilon)t_*^\alpha)}&\le \lambda^*-\varepsilon-a+\frac{bE_{\alpha}(-(\lambda^*-\varepsilon)(t_*-\tau_1)^\alpha)}{E_{\alpha}(-(\lambda^*-\varepsilon)t_*^\alpha)}\notag\\
 &\hspace{-1cm}+\frac{ec}{E_{\alpha}(-(\lambda^*-\varepsilon)\tau_3^\alpha)-d}\times\frac{1}{E_{\alpha}(-(\lambda^*-\varepsilon)t_*^\alpha)}+\frac{-a\gamma_1+b\gamma_1+c\gamma_2+f}{(M+\varepsilon)E_{\alpha}(-(\lambda^*-\varepsilon)t_*^\alpha)}\notag\\
	&\le \lambda^*-\varepsilon-a+\frac{b}{E_{\alpha}(-(\lambda^*-\varepsilon)\tau_1^\alpha)}+\frac{ec}{E_{\alpha}(-(\lambda^*-\varepsilon)\tau_3^\alpha)-d}\times\frac{1}{E_{\alpha}(-(\lambda^*-\varepsilon)\tau_2^\alpha)}\notag\\
 &=h(\lambda^*-\varepsilon)\notag\\
 &<h(\lambda^*)=0,\notag
\end{align}
a contradiction.

\textbf{Case I.3.2:} $\min\left\{\tau_1,\tau_2\right\}\le t_*\le\max\left\{\tau_1,\tau_2\right\}$, $\tau_2<\tau_1$. We have 	
\begin{equation*}
	\begin{cases}
		\displaystyle\sup_{0\le s_1\le\tau_1}u(t_*-s_1)&\leq M+\varepsilon+\gamma_1,\\
		\displaystyle\sup_{0\le s_2\le\tau_2}v(t_*-s_2)&\leq \displaystyle\frac{(M+\varepsilon)c}{E_{\alpha}(-\lambda^*\tau_3^\alpha)-d}E_{\alpha}(-\lambda^*(t_*-\tau_2)^\alpha)+\gamma_2.
	\end{cases}
\end{equation*}
With the help of Lemma \ref{ML1} and Lemma \ref{Est-ML}, then 
\begin{align*}
	^{\!C}D^{\alpha}_{0^+}z(t_*)&\le (M+\varepsilon)\left(\lambda^*-\varepsilon-a\right)E_{\alpha}(-(\lambda^*-\varepsilon)t_*^\alpha)+b(M+\varepsilon)\\
 &\hspace{1cm}+\frac{(M+\varepsilon)ec}{E_{\alpha}(-(\lambda^*-\varepsilon)\tau_3^\alpha)-d}E_{\alpha}(-(\lambda^*-\varepsilon)(t_*-\tau_2)^\alpha)-a\gamma_1+b\gamma_1+c\gamma_2+f.
 \end{align*}
 Hence,
 \begin{align}
 \frac{^{\!C}D^{\alpha}_{0^+}z(t_*)}{(M+\varepsilon)E_{\alpha}(-(\lambda^*-\varepsilon)t_*^\alpha)}
	&\le \lambda^*-\varepsilon-a+\frac{b}{E_{\alpha}(-(\lambda^*-\varepsilon)t_*^\alpha)}\notag\\
 &\hspace{-2cm}+\frac{ec}{E_{\alpha}(-(\lambda^*-\varepsilon)\tau_3^\alpha)-d}\times\frac{E_{\alpha}(-(\lambda^*-\varepsilon)(t_*-\tau_2)^\alpha)}{E_{\alpha}(-(\lambda^*-\varepsilon)t_*^\alpha)}+\frac{-a\gamma_1+b\gamma_1+c\gamma_2+f}{(M+\varepsilon)E_{\alpha}(-(\lambda^*-\varepsilon)t_*^\alpha)}\notag\\
 &\le \lambda^*-\varepsilon-a+\frac{b}{E_{\alpha}(-(\lambda^*-\varepsilon)\tau_1^\alpha)}+\frac{ec}{E_{\alpha}(-(\lambda^*-\varepsilon)\tau_3^\alpha)-d}\times\frac{1}{E_{\alpha}(-(\lambda^*-\varepsilon)\tau_2^\alpha)}\notag\\
	&=h(\lambda^*-\varepsilon)\notag\\
 &<h(\lambda^*)=0,\notag
\end{align}
a contradiction.

We now assume that the counterfactual $\textup{(ii)}$ is true. Let $\gamma_1,\gamma_2$ satisfies
$ -\gamma_2+c\gamma_1+d\gamma_2=0$.
From \eqref{E0_2}, we see
\begin{align}\label{est-v}
	\frac{(M+\varepsilon)c}{E_{\alpha}(-(\lambda^*-\varepsilon)\tau_3^\alpha)-d}E_{\alpha}(-(\lambda^*-\varepsilon)t_*^\alpha)+\gamma_2&\le (M+\varepsilon)cE_{\alpha}(-(\lambda^*-\varepsilon)t_*^\alpha)+c\gamma_1+d\sup_{0\le s_3\le \tau_3}v(t_*-s_3).
\end{align}

\textbf{Case II.1:} $t_*>\tau_3$. By \eqref{contra_5},   $\displaystyle\sup_{0\le s_3\le \tau_3}v(t_*-s_3)\leq\frac{(M+\varepsilon)c}{E_{\alpha}(-(\lambda^*-\varepsilon)\tau_3^\alpha)-d}E_{\alpha}(-(\lambda^*-\varepsilon)(t_*-\tau_3)^\alpha)+\gamma_2$, which together with \eqref{est-v} implies that
\begin{align*}\frac{E_{\alpha}(-(\lambda^*-\varepsilon)t_*^\alpha)}{E_{\alpha}(-(\lambda^*-\varepsilon)\tau_3^\alpha)-d}&\le E_{\alpha}(-(\lambda^*-\varepsilon)t_*^\alpha)+\frac{d}{E_{\alpha}(-(\lambda^*-\varepsilon)\tau_3^\alpha)-d}E_{\alpha}(-(\lambda^*-\varepsilon)(t_*-\tau_3)^\alpha)
+\frac{-\gamma_2+c\gamma_1+d\gamma_2}{(M+\varepsilon)c}.\end{align*}
Thus,
\begin{align}
\frac{1}{E_{\alpha}(-(\lambda^*-\varepsilon)\tau_3^\alpha)-d}&\le 1+\frac{d}{E_{\alpha}(-(\lambda^*-\varepsilon)\tau_3^\alpha)-d}\times\frac{E_{\alpha}(-(\lambda^*-\varepsilon)(t_*-\tau_3)^\alpha)}{E_{\alpha}(-(\lambda^*-\varepsilon)t_*^\alpha)}\notag\\
&\le 1+\frac{d}{E_{\alpha}(-(\lambda^*-\varepsilon)\tau_3^\alpha)-d}\times\frac{1}{E_{\alpha}(-(\lambda^*-\varepsilon)\tau_3^\alpha)}\;\;\left(\text{by Lemma \ref{Est-ML}}\right)\notag\\
&=\frac{\left[E_{\alpha}(-(\lambda^*-\varepsilon)\tau_3^\alpha)\right]^2-dE_{\alpha}(-(\lambda^*-\varepsilon)\tau_3^\alpha)+d}{E_{\alpha}(-(\lambda^*-\varepsilon)\tau_3^\alpha)-d}\notag\\
&=\frac{\left[E_{\alpha}(-(\lambda^*-\varepsilon)\tau_3^\alpha)-1\right]\left[E_{\alpha}(-(\lambda^*-\varepsilon)\tau_3^\alpha)-d\right]+E_{\alpha}(-(\lambda^*-\varepsilon)\tau_3^\alpha)}{E_{\alpha}(-(\lambda^*-\varepsilon)\tau_3^\alpha)-d}\notag\\
&<\frac{1}{E_{\alpha}(-(\lambda^*-\varepsilon)\tau_3^\alpha)-d},\notag
\end{align}
a contradiction.

\textbf{Case II.2:} $t_*\le\tau_3$. From \eqref{icd_2}, \eqref{contra_4}, \eqref{contra_5}, it leads to $\displaystyle\sup_{0\le s_3\le \tau_3}v(t_*-s_3)\leq \frac{(M+\varepsilon)c}{E_{\alpha}(-\lambda^*\tau_3^\alpha)-d}+\gamma_2$, which combines with the estimate \eqref{est-v} shows that
\begin{align*}\frac{(M+\varepsilon)c}{E_{\alpha}(-(\lambda^*-\varepsilon)\tau_3^\alpha)-d}E_{\alpha}(-(\lambda^*-\varepsilon)t_*^\alpha)+\gamma_2&\le (M+\varepsilon)cE_{\alpha}(-(\lambda^*-\varepsilon)t_*^\alpha)+c\gamma_1
+d\frac{(M+\varepsilon)c}{E_{\alpha}(-(\lambda^*-\varepsilon)\tau_3^\alpha)-d}+d\gamma_2.
\end{align*}
Hence, 
\begin{align}
	\frac{1}{E_{\alpha}(-\lambda^*\tau_3^\alpha)-d}&\le 1+\frac{d}{E_{\alpha}(-\lambda^*\tau_3^\alpha)-d}\times\frac{1}{E_{\alpha}(-\lambda^*t_*^\alpha)}+\frac{-\gamma_2+c\gamma_1+d\gamma_2}{(M+\varepsilon)cE_{\alpha}(-(\lambda^*-\varepsilon)t_*^\alpha)}\notag\\
	&\le 1+\frac{d}{E_{\alpha}(-\lambda^*\tau_3^\alpha)-d}\times\frac{1}{E_{\alpha}(-\lambda^*\tau_3^\alpha)}\;\; \left(\text{by Lemma \ref{ML1}}\right)\notag\\
	&<\frac{1}{E_{\alpha}(-\lambda^*\tau_3^\alpha)-d}, \;\;\left(\text{as discussed above}\right)\notag
\end{align}
a contradiction.

In short, based on all the observations obtained above, by choosing
 \begin{align}
  \gamma_1= \frac{(1-d)f}{(1-d)(a-b)-ce}, \quad
  \gamma_2=\frac{cf}{(1-d)(a-b)-ce},\label{coef2}
 \end{align}
we conclude that the estimates \eqref{est_hala_add_1}--\eqref{est_hala_add_2} are true for all $t\geq 0$. Notice that $\gamma_1,\gamma_2$ chosen as \eqref{coef2} is the unique solution of the system. The desired estimates \eqref{est_hala_1}--\eqref{est_hala_2} by letting $\varepsilon\to 0$. The proof is complete.
\end{proof}

\begin{remark}
Some special forms of Theorem \ref{Theorem-Hala} have been published in the literature. Indeed, considering system \eqref{E0_1}--\eqref{E0_2}, the case when $e=d=0$ has been discussed in \cite[Lemma 2.3]{Wang_15} and \cite[Lemma 4]{Wang_19}, the case $b=0$ has been analyzed in \cite[Lemma 4]{Wang}.    
\end{remark}

\begin{remark}
Theorem \ref{Theorem-Hala} presents a generalized and strengthened version of \cite[Lemma 4]{Wang}. 
Compared with previous work, in this paper, we construct a new parameter equation $h(\lambda)$ with respect to $\lambda$ in \eqref{eq:hfun} to more clearly and concisely prove the strict negativity of parameter $\lambda^*$, which ensures the Mittag-Leffler stability of the time fractional order model, i.e. the optimal long-term algebraic decay rate.
\end{remark}

\section{Mittag-Leffler stability of fractional-order linear coupled systems with delays}

This part analyzes the asymptotic behavior of solutions to fractional-order linear coupled systems with delays, which may not necessarily be positive. Our approach combines the Halanay-type fractional inequality  with the positive representation method developed by Iuliis et al. \cite{Luliis}. 
The system under consideration is as follows:
\begin{align} \label{Eq}
		\begin{cases}
		^{\!C}D^{{\alpha}}_{0+}x(t)&=Ax(t)+Bx(t-\tau_1)+Ey(t-\tau_2),\; t\in (0,\infty),\\
			y(t)&= Cx(t)+Dy(t-\tau_3),\; t\in [0,\infty),\\
x(\cdot)&=\varphi(\cdot),\;y(\cdot)=\psi(\cdot)\;\text{on}\; [-\tau,0],
       		\end{cases}	
	\end{align}
where $A,\ B\in \R^{d\times d},\ E\in \R^{d\times n},\ C\in \R^{n\times d},\ D\in \R^{n\times n}$, 
$\varphi [-r,0]\rightarrow \R^d,\ \psi [-r,0]\rightarrow \R^n$ are given continuous functions satisfying the compatibility condition 
\begin{align} \label{condK}
\textup{(K):}\;\; C\varphi(0)+D\psi(-\tau_3)=\psi(0),
\end{align}
where $\tau_1,\tau_2,\tau_3$ are positive delays and $\tau=\max\left\{\tau_1,\tau_2,\tau_3\right\}$. 

\begin{lemma}\label{exist_sol}
 The system \eqref{Eq} has a unique solution  $\Phi(\cdot;\varphi,\psi)$ on $[-r,\infty)$, where $$\Phi(t;\varphi,\psi)=\left(\begin{array}{cc}
		x(t;\varphi,\psi) \\ y(t;\varphi,\psi)\end{array}\right) =\left(\begin{array}{cc}
		x_1(t;\varphi,\psi) \\ \vdots \\ x_d(t;\varphi,\psi) \\y_1(t;\varphi,\psi) \\ \vdots \\ y_n(t;\varphi,\psi)\end{array}\right).$$ 
\end{lemma}
\begin{proof}
The proof is based on \cite[Theorem 3.2]{Tuan-Thinh} and \cite[Remark 3.4]{Tuan-Thinh}.
\end{proof}

\begin{proposition} \cite[Theorem 4.2]{Tuan-Thinh}\label{Pro-positive}
Suppose that $A$ is Metzler, $B,E,C,D$ are nonnegative.
Then, the system \eqref{Eq} is positive.
\end{proposition} 

We now recall some basic concepts from \cite{Luliis}. For a matrix (or vector) $M=(m_{ij})\in\R^{k\times l}$, the symbols $M^+=(m^+_{ij}),\ M^-=(m^-_{ij})$ denote the
component-wise positive and negative parts of $M$, i.e.,
\[m^+_{ij}=\begin{cases}
    m_{ij}\ &\text{if}\ m_{ij}\ge0,\\
    0\ &\text{if}\ m_{ij}<0,
\end{cases} \ \text{and} \ m^-_{ij}=\begin{cases}
    -m_{ij}\ &\text{if}\ m_{ij}\le0,\\
    0\ &\text{if}\ m_{ij}>0.
\end{cases}\]
 Moreover, $|M|$ denotes the component-wise absolute value of $M$. We easily observe that $M=M^+-M^-\ \text{and}\ |M|=M^++M^-.$
 
 \begin{definition} \cite[Definition 2]{Luliis}
     A positive representation of a vector $w\in\R^k$ is any vector $\Tilde{w}\in\R_{\geq 0}^{2k}$ such that \[w=\Delta_k \Tilde{w},\]
     here $\Delta_k=(\delta_{ij})_{k\times{2k}}$ with $\delta_{ij}=\begin{cases}
         1\ &\text{if}\ i=j,\\
         -1\ &\text{if}\ j=i+k,\\
         0\ &\text{if} \ i\ne j\ \text{and}\ j\ne i+k.
     \end{cases}$
     The min-positive representation of a vector $w\in\R^k$ is the nonnegative vector $\pi(w)\in\R_{\geq 0}^{2k}$ as below:
     $\pi(w)=\begin{pmatrix}
         w^+\\ w^-
     \end{pmatrix}.$
     The min-positive representation of a matrix $M\in\R^{k\times l}$ is the nonnegative matrix $\Pi(M)\in\R_{\geq 0}^{2k\times2l}$ given as
     $\Pi(M)=\begin{pmatrix}
         M^+& M^-\\ M^- & M^+
     \end{pmatrix}.$
The min-Metzler representation of a matrix $A\in\R^{k\times k}$ is the
Metzler matrix $\Gamma(A)\in\R_{\geq 0}^{2k\times 2k}$ defined by
\[\Gamma(A)=\begin{pmatrix}
         \text{diag}(A)+(A-\text{diag}(A))^+& (A-\text{diag}(A))^-\\ (A-\text{diag}(A))^- & \text{diag}(A)+(A-\text{diag}(A))^+
     \end{pmatrix}.\]
 \end{definition}
 \begin{remark}
     For any $x\in\R^k$ and matrices $M\in\R^{k\times l},\ A\in\R^{k\times k}$ the following properties hold true:
     \begin{itemize}
         \item [a)] $x=\Delta_k\pi(x),$
         \item [b)] $\Delta_k\Pi(M)=M\Delta_k$, so that $\Delta_k\Pi(M)\pi(x)=Mx,$ 
         \item [c)] $\Delta_k\Gamma(A)=A\Delta_k$, so that $\Delta_k\Gamma(A)\pi(x)=Ax.$ 
     \end{itemize}
 \end{remark}
 
 Based on \cite[Definition 3]{Luliis}, we propose the following definition for fractional-order differential systems.
 
 \begin{definition}
     A positive representation (PR) of a fractional delay differential system as in \eqref{Eq} is a positive system in the form:
     \begin{equation*}
	\begin{cases}
		^{\!C}D^{\alpha}_{0^+}\tilde{x}(t)&=\tilde{A}\tilde{x}(t)+\tilde{B}\tilde{x}(t-\tau_1)+\tilde{E}\tilde{y}(t-\tau_2),\ \forall t\ge0,\\
		\tilde{y}(t)&= \tilde{C}\tilde{x}(t)+\tilde{D}\tilde{y}(t-\tau_3),\; \forall t>0,\\
\tilde{x}(\cdot)&=\tilde{\varphi}(\cdot),\;\tilde{y}(\cdot)=\tilde{\psi}(\cdot)\;\text{on}\; [-\tau,0],
	\end{cases}
\end{equation*}
 accompanied by four continuous transformations $\{T_X^f,\ T_X^b,\ T_Y^f,\ T_Y^b\}$,
\begin{align*}
    T_X^f:\R^d\rightarrow \R_{\geq 0}^{2d},\ T_X^b:\R_{\geq 0}^{2d}\rightarrow\R^d,
    T_Y^f:\R^n\rightarrow \R_{\geq 0}^{2n},\ T_Y^b:\R_{\geq 0}^{2n}\rightarrow\R^n,
\end{align*}
such that for all  $\varphi \in C([-r,0];\R^d)$, $\psi\in C([-r,0];\R^n)$ satisfying the compatibility condition (K), the following implication holds:
\begin{align*}
    \begin{cases}
        \tilde{\varphi}(s)&=T_X^f(\varphi(s))\\
        \tilde{\psi}(s)&=T_Y^f(\psi(s))
    \end{cases} ,\ \forall s\in[-r,0] \Longrightarrow \begin{cases}
        x(t)&=T_X^b(\tilde{x}(t))\\
         y(t)&=T_Y^b(\tilde{y}(t))
    \end{cases}
   ,\ \forall t\geq 0.
\end{align*}
 \end{definition}
 
 \begin{theorem} \label{Theorem-IPR}
     Consider a delay differential system \eqref{Eq} with the compatibility condition (K). A positive system  
      \begin{equation} \label{Equation-main-IPR}
	\begin{cases}
		^{\!C}D^{\alpha}_{0^+}\tilde{x}(t)&=\tilde{A}\tilde{x}(t)+\tilde{B}\tilde{x}(t-\tau_1)+\tilde{E}\tilde{y}(t-\tau_2),\ \forall t>0,\\
  \tilde{y}(t)&= \tilde{C}\tilde{x}(t)+\tilde{D}\tilde{y}(t-\tau_3),\; \forall t\ge0,\\
\tilde{x}(\cdot)&=\tilde{\varphi}(\cdot),\;\tilde{y}(\cdot)=\tilde{\psi}(\cdot)\;\text{on}\; [-\tau,0],
	\end{cases}
\end{equation}
with $\tilde{A}=\Gamma(A),\ \tilde{B}=\Pi(B),\ \tilde{E}=\Pi(E),\ \tilde{C}=\Pi(C),\ \tilde{D}=\Pi(D)$, accompanied by the four continuous transformations, defined as:
\begin{align*}
    \tilde{x}=T_X^f(x)=\pi(x),\ x=T_X^b(\tilde{x})=\Delta_d\tilde{x},
    \tilde{y}=T_Y^f(y)=\pi(y),\ y=T_Y^b(\tilde{y})=\Delta_n\tilde{y},
\end{align*}
which is a PR of \eqref{Eq}.
 \end{theorem}
 
 \begin{proof}
     Since $\tilde{A}$ is a Metzler matrix, $\tilde{B},\ \tilde{E},\ \tilde{C},\ \tilde{D}\ $ are nonnegative matrices, so by Proposition  \ref{Pro-positive}, \eqref{Equation-main-IPR} is positive.
     Define
     \begin{align*}
         \hat{x}(t)=T_X^b(\tilde{x}(t))=\Delta_d\tilde{x}(t),\ t\ge-\tau; \quad
         \hat{y}(t)=T_Y^b(\tilde{y}(t))=\Delta_n\tilde{y}(t),\ t\ge-\tau.
     \end{align*}
     
     \begin{itemize}
         \item If $t\in[-\tau,0]$, then
         \begin{align*}
             \hat{x}(t)=\Delta_d\tilde{x}(t)=\Delta_d\tilde{\varphi}(t)=\varphi(t),\quad
             \hat{y}(t)=\Delta_n\tilde{y}(t)=\Delta_n\tilde{\psi}(t)=\psi(t).
         \end{align*}
         Besides that, from (K), we have
         \begin{align*}
         \Delta_n\tilde{\psi}(0)&=C\Delta_d\tilde{\varphi}(0)+D\Delta_n\tilde{\psi}(-\tau_3)\\
         &=\Delta_n \Pi(C)\tilde{\varphi}(0)+\Delta_n\Pi(D)\tilde{\psi}(-\tau_3)\\
         &=\Delta_n \bigg[\tilde{C}\tilde{\varphi}(0)+\tilde{D}\tilde{\psi}(-\tau_3)\bigg],
         \end{align*}
         this implies $\tilde{C}\tilde{\varphi}(0)+\tilde{D}\tilde{\psi}(-\tau_3)=\tilde{\psi}(0).$
         \item If $t>0$, then
         \begin{align*}
             ^{\!C}D^{{\alpha}}_{0^+}{\hat{x}}(t)&= {^{\!C}D^{{\alpha}}_{0^+}}\Delta_d{\tilde{x}}(t)=\Delta_d{^{\!C}D^{{\alpha}}_{0^+}}{\tilde{x}}(t)\\
             &=\Delta_d \tilde{A}\tilde{x}(t)+\Delta_d \tilde{B}\tilde{x}(t-\tau_1)+\Delta_d \tilde{E}\tilde{y}(t-\tau_2)\\
             &=\Delta_d \Gamma{(A)}\pi({x}(t))+\Delta_d \Pi{(B)}\pi({x}(t-\tau_1))+\Delta_d \Pi{(E)}\pi({y}(t-\tau_2))\\
             &=A {x}(t)+B{x}(t-\tau_1)+E{y}(t-\tau_2),\\
             \hat{y}(t)&=\Delta_n\tilde{y}(t)=\Delta_n\tilde{C}\tilde{x}(t)+\Delta_n\tilde{D}\tilde{y}(t-\tau_3)\\
             &=\Delta_n \Pi{(C)}\pi({x}(t))+\Delta_n \Pi{(D)}\pi({y}(t-\tau_3))\\
             &=Cx(t)+Dy(t-\tau_3).
         \end{align*}
     \end{itemize}
     Thus $\hat{x}(t)$, $\hat{y}(t)$ is also a solution of the system \eqref{Eq}. Due to the fact the system \eqref{Eq} has a unique solution, it leads to that $\hat{x}(t)=x(t), \hat{y}(t)=y(t), \forall t\in[-\tau,+\infty),$
     which finishes the proof.
 \end{proof}
 
 Our main contribution is the following theorem.
\begin{theorem} \label{Theorem-app}
    Consider the system \eqref{Eq}. Let $x(\cdot)=(x_1(\cdot),\dots,x_d(\cdot))^{\rm T}$, $y(\cdot)=(y_1(\cdot),\dots,y_n(\cdot))^{\rm T}$ be its solution. Define
    \[\tilde{A}:=\Gamma(A),\ \tilde{B}:=\Pi(B),\ \tilde{E}:=\Pi(E),\ \tilde{C}:=\Pi(C),\ \tilde{D}:=\Pi(D).\]
    Assume that $a_0,\ b_0,\ e_0,\ c_0,\ d_0$ are five real numbers with 
\begin{align*} a_0=-\max_{j\in\{1,\dots,2d\}}\sum_{i=1}^{2d}\tilde{a}_{ij},\ b_0=\max_{j\in\{1,\dots,2d\}}\sum_{i=1}^{2d}\tilde{b}_{ij},\ e_0=\max_{j\in\{1,\dots,2n\}}\sum_{i=1}^{2d}\tilde{e}_{ij},
c_0=\max_{j\in\{1,\dots,2d\}}\sum_{i=1}^{2n}\tilde{c}_{ij},\ d_0=\max_{j\in\{1,\dots,2n\}}\sum_{j=1}^{2n}\tilde{d}_{ij}
\end{align*} 
satisfying that
$0\le d_0<1,b_0\ge 0,c_0\ge 0,a_0>b_0+\displaystyle\frac{e_0c_0}{1-d_0}>0.$ Then, for any $\varphi\in C([-\tau,0];\R^d),\ \psi\in C([-\tau,0];\R^n)$ such that  the condition \textup{(K)} is verified,  
there are two positive constants $M>0, \lambda^*>0$ satisfying
\begin{align*} 
	|x_i(t)|&\leq ME_{\alpha}(-\lambda^*t^\alpha),\;\text{for}\;t\geq 0,\;i=1,\dots,d,\\
|y_i(t)|&\leq \frac{Mc_0}{E_{\alpha}(-\lambda^*\tau_3^\alpha)-d_0}E_{\alpha}(-\lambda^*t^\alpha),\;\text{for}\;t\geq 0,\;i=1,\dots,n,
\end{align*}
where $\lambda^*$ is a unique positive solution of the following equation 
\[h_{1}(\lambda):=\lambda-a_0+\frac{b_0}{E_{\alpha}(-\lambda \tau_1^\alpha)}+\frac{e_0c_0}{E_{\alpha}(-\lambda \tau_3^\alpha)-d_0}\times\frac{1}{E_{\alpha}(-\lambda \tau_2^\alpha)}=0.\]
\end{theorem}

\begin{proof}
For convenience in presentation, we denote $\tilde{x}(\cdot)=(\tilde{x}_1(\cdot),\dots,\tilde{x}_{2d}(\cdot))^{\rm T}$, $\tilde{y}(\cdot)=(\tilde{y}_1(\cdot),\dots,\tilde{y}_{2n}(\cdot))^{\rm T}$ as the solution of the system \eqref{Equation-main-IPR}. Due to the fact that the system \eqref{Equation-main-IPR} is positive (it follows from \cite[Theorem 4.2]{Tuan-Thinh}), we have
$\tilde{x}_i(t)\geq 0$ and $\tilde{y}_1(\cdot)\geq 0$ for $t\geq -\tau,\;i=1,\dots,2n.$
 Define $\tilde{X}(t):=\tilde{x}_1(t)+\tilde{x}_2(t)+\cdots+\tilde{x}_{2d}(t)$ and $\tilde{Y}(t):=\tilde{y}_1(t)+\tilde{y}_2(t)+\cdots+\tilde{y}_{2n}(t),\;\forall t\in [-\tau,+\infty)$.
It is obvious to see that 
		\begin{align}
			{^{\!C}D^{\alpha}_{0^+}}\tilde{X}(t)&={^{\!C}D^{\alpha}_{0^+}}\tilde{x}_1(t)+{^{\!C}D^{\alpha}}_{0^+}\tilde{x}_2(t)+\cdots+{^{\!C}D^{\alpha}_{0^+}}\tilde{x}_{2d}(t)\notag\\
			&\hspace{-1.7cm}=\sum_{j=1}^{2d}\tilde{a}_{1j}\tilde{x}_j(t)+\sum_{j=1}^{2d}\tilde{b}_{1j}\tilde{x}_j(t-\tau_1)+\sum_{j=1}^{2n} \tilde{e}_{1j}\tilde{y}_j(t-\tau_2)+\sum_{j=1}^{d}\tilde{a}_{2j}(t)\tilde{x}_j(t)+\sum_{j=1}^{d}\tilde{b}_{2j}\tilde{x}_j(t-\tau_2)\notag\\
            &+\sum_{j=1}^{2n} \tilde{e}_{2j}\tilde{y}_j(t-\tau_2)+\cdots+\sum_{j=1}^{2d}\tilde{a}_{2dj}\tilde{x}_j(t)+\sum_{j=1}^{2d}\tilde{b}_{2dj}\tilde{x}_j(t-\tau_1)+\sum_{j=1}^{2n} \tilde{e}_{2dj}\tilde{y}_j(t-\tau_2)\notag\\
			&\hspace{-1.7cm}=\sum_{i=1}^{2d}\tilde{a}_{i1}\tilde{x}_1(t)+\sum_{i=1}^{2d}\tilde{a}_{i2}\tilde{x}_2(t)+\cdots+\sum_{i=1}^{2d}\tilde{a}_{i2d}\tilde{x}_{2d}(t)\notag\\
			&\hspace{-0.7cm}+\sum_{i=1}^{2d}\tilde{b}_{i1}x_1(t-\tau_1)+\sum_{i=1}^{2d}\tilde{b}_{i2}x_2(t-\tau_1)+\cdots+\sum_{i=1}^{2d}b_{i2d}x_d(t-\tau_1)\notag\\
            &\hspace{-0.7cm}+\sum_{i=1}^{d}\tilde{e}_{i1}y_1(t-\tau_2)+\sum_{i=1}^{2d}\tilde{e}_{i2}y_2(t-\tau_2)+\cdots+\sum_{i=1}^{2d}\tilde{b}_{in}y_{2n}(t-\tau_2)\notag\\
			&\hspace{-1.7cm}\le\bigg(\max_{1\leq j\leq 2d}\sum_{i=1}^{2d}\tilde{a}_{ij}\bigg)X(t)+\bigg(\max_{1\leq j\leq 2d}\sum_{i=1}^{2d}\tilde{b}_{ij}\bigg)X(t-\tau_1)+\bigg(\max_{1\leq j\leq 2n}\sum_{i=1}^{2d}\tilde{e}_{ij}\bigg)Y(t-\tau_2),\;\forall t>0.\label{dk1}
		\end{align}  
        Moreover, 
        \begin{equation}\label{dk2}
         \tilde{Y}(t)\leq \Big(\max_{1\leq j\leq 2d}\sum_{i=1}^{2d}\tilde{c}_{ij}\Big) \tilde{X}(t)+  \Big(\max_{1\leq j\leq 2n}\sum_{i=1}^{2d}\tilde{d}_{ij}\Big) \tilde{Y}(t-\tau_3),\;\forall t\geq 0.
        \end{equation} 
        Let
		\[a_0:=-\max_{j\in\{1,\dots,2d\}}\sum_{i=1}^{2d}\tilde{a}_{ij},\;b_0:=\max_{j\in\{1,\dots,2d\}}\sum_{i=1}^{2d}\tilde{b}_{ij},\;e_0:=\max_{j\in\{1,\dots,2n\}}\sum_{i=1}^{2d}\tilde{e}_{ij},
        c_0:=\max_{1\leq j\leq 2d}\sum_{i=1}^{2d}\tilde{c}_{ij}\;\text{and}\; d_0:=\max_{1\leq j\leq 2n}\sum_{i=1}^{2d}\tilde{d}_{ij}.
        \]
        By combining \eqref{dk1}, \eqref{dk2}, Theorem \ref{Theorem-Hala} and the assumptions stated in the theorem on the coefficients $a_0,b_0,c_0,d_0,e_0$, we conclude  that
        \begin{align} 
	\tilde{X}(t)\leq ME_{\alpha}(-\lambda^*t^\alpha), \quad
\tilde{Y}(t)\leq \frac{Mc_0}{E_{\alpha}(-\lambda^*\tau_3^\alpha)-d_0}E_{\alpha}(-\lambda^*t^\alpha)\label{tg2}
\end{align}
for all $t\geq 0$. Here, 
$\lambda^*$ is a unique positive solution of the following equation 
\[\lambda-a_0+\frac{b_0}{E_{\alpha}(-\lambda \tau_1^\alpha)}+\frac{e_0c_0}{E_{\alpha}(-\lambda \tau_3^\alpha)-d_0}\times\frac{1}{E_{\alpha}(-\lambda \tau_2^\alpha)}=0\]
and $M>0$ satisfies 
$\sup_{0\le s\le \tau}\tilde{X}(s-\tau)\leq M,\,\sup_{0\le s\le \tau}\tilde{Y}(s-\tau)\leq \frac{Mc_0}{E_{\alpha}(-\lambda^* \tau_3^\alpha)-d_0}.$
 On the other hand, by Theorem \ref{Theorem-IPR},
 \begin{align}
   x_i(t)=\tilde{x}_i-\tilde{x}_{i+d}(t),\,  
   y_i(t)=\tilde{y}_i-\tilde{y}_{i+n}(t)\;\text{for}\;t\geq -\tau,\;i=1,\dots,n.
 \end{align}
 These imply the required results.  The proof is complete.
\end{proof}

Theorem \ref{Theorem-app} establishes a novel approach to analyzing linear fractional-order systems outside the realm of positive system theory. To the best of our knowledge, it is the first result of its kind.
To close this section, we provide a concrete example and a numerical simulation to validate Theorem \ref{Theorem-app}.

\begin{example}
   Consider the fractional-order linear coupled system defined as
\begin{align} \label{Eq-Exam}
		\begin{cases}
		^{\!C}D^{{\alpha}}_{0+}x(t)&=Ax(t)+Bx(t-\tau_1)+Ey(t-\tau_2),\; t\in (0,\infty),\\
			y(t)&= \hspace{0.3cm}Cx(t)+Dy(t-\tau_3),\; t\in [0,\infty),\\
x(\cdot)&=\varphi(\cdot),\;y(\cdot)=\psi(\cdot)\;\text{on}\; [-\tau,0],
       		\end{cases}	
	\end{align}
 where ${\alpha}=0.6$, $\tau_1=1$, $\tau_2=0.6$, $\tau_3=0.9,$
	\begin{align*}
		A&=	\begin{pmatrix}
			-2.2 & -0.9\\
			0.7 & -2.3
		\end{pmatrix},\
		B=	\begin{pmatrix}
			-0.2 & 0.1 \\
			-0.2 & 0.5
		\end{pmatrix},\
		E=	\begin{pmatrix}
			0.2 & -0.1 & 0.1\\
			-0.2 & 0.3 & -0.1			
		\end{pmatrix},\\
		C&=	\begin{pmatrix}
			-0.1 & 0.2 \\
			0.1 & -0.3\\
               - 0.3 & -0.2
		\end{pmatrix},\
		D=	\begin{pmatrix}
			-0.1 & 0.2   & 0.2\\
			0.1 & -0.2   & -0.2\\
			-0.2 & -0.1 & 0.2
		\end{pmatrix}.
	\end{align*}
	By a simple calculation, we obtain 
 \begin{align*}
		\tilde{A}&=	\begin{pmatrix}
			-2.2 & 0&0&0.9\\
			0.7 & -2.3&0&0\\
            0&0.9&-2.2&0\\
            0&0&0.7&-2.3
		\end{pmatrix},\
		\tilde{B}=	\begin{pmatrix}
			0 & 0.1&0.2&0 \\
			0 & 0.5&0.2&0\\
            0.2 & 0&0&0.1\\
            0.2 & 0&0&0.5
		\end{pmatrix},
		\tilde{E}=	\begin{pmatrix}
			0.2 & 0 & 0.1&0&0.1&0\\
			0 & 0.3 & 0&0.2&0&0.1\\
            0&0.1&0&0.2 & 0 & 0.1\\
            0.2&0&0.1&0 & 0.3 & 0
		\end{pmatrix},\\
		\tilde{C}&=	\begin{pmatrix}
			0 & 0.2&0.1&0 \\
			0.1 & 0&0&0.3\\
                0 & 0&0.3&0.2\\
                0.1&0&0 & 0.2\\
                0&0.3&0.1 & 0\\
                0.3&0.2&0&0
		\end{pmatrix},\
		\tilde{D}=	\begin{pmatrix}
			0 & 0.2   & 0.2&0.1&0&0\\
			0.1 & 0   & 0&0&0.2&0.2\\
			0 & 0 & 0.2&0.2&0.1&0\\
            0.1&0&0&0 & 0.2   & 0.2\\
            0&0.2&0.2&0.1 & 0   & 0\\0.2&0.1&0&0 & 0 & 0.2
		\end{pmatrix}.
	\end{align*}
 This implies 
 \[a_0=1.4,\ b_0=0.6,\ c_0=0.7,\ d_0=0.6,\ e_0=0.4. \]
 It is clear that $0\le d_0<1,\ b_0\ge 0,\ c_0\ge 0,\ a_0=1.4>b_0+\displaystyle\frac{e_0c_0}{1-d_0}=1.3>0$. According to Theorem \ref{Theorem-app}, we conclude that the system \eqref{Eq-Exam} is globally attractive. 
\end{example}

\section{Contractivity and dissipativity of fractional neutral FDEs}
Another important application of fractional coupled Halanaly inequality is to establish the dissipativity and contractility of  fractional neutral functional differential equations (F-NFDEs). In \cite{Wang}, similar results have been obtained, but the parameter $\lambda^*$ involved in Mittag-Leffler function depends on a very complex equation, which make it difficult to prove its strict negativity.
We present a new and more concise analysis in this article.
Consider the initial value problems of F-NFDEs of Hale type
\begin{equation} \label{eq1}
\begin{split}
^{C}{D}_{0}^{\alpha}\left[ y(t)-Ny(t-\tau)\right]=f\left(t, y(t), y(t-\tau), \int_{t-\tau}^{t}g(t, \xi, y(\xi))d\xi \right),~~~t>0,\\
\end{split}
\end{equation}
subject to the initial function $y(t)=\varphi(t)$ for $t\in[-\tau, 0]$, $N\in \mathbb{R}^{n\times n}$ denotes a constant matrix with $2\|N\|^2 < 1$ (for convenience, we also use the notation $\|\cdot\|$ to represent the matrix norm on $\mathbb{R}^{n \times n}$). 
Throughout this paper, we assume that the problem (\ref{eq1}) has a unique continuous solution for all $t\geq0$. 
We consider that the nonlinear function $f$ satisfies two different kinds of structural assumptions.

\begin{itemize}
 \item
The continuous mapping $f: [0,+\infty)\times \mathbb{R}^{n} \times \mathbb{R}^{n}\times \mathbb{R}^{n}  \rightarrow \mathbb{R}^{n}$ 
satisfies the one-sided Lipschitz  condition
\begin{align}\label{eq2_ineq}
\langle u_{1}-u_{2}-N(v_{1}-v_{2}), f(t, u_{1}, v_{1}, w_{1})-f(t, u_{2}, v_{2}, w_{2})\rangle
\leq  a_1\|u_{1}-u_{2}\|^2+ b_1 \|v_{1}-v_{2}\|^2+c_1 \|w_{1}-w_{2}\|^2,
\end{align}
and $ g: [0,+\infty)\times [0,+\infty)\times \mathbb{R}^{n}  \rightarrow \mathbb{R}^{n}$ satisfies the Lipschitz condition
\begin{equation}\label{eq3_ineq}
\|g(t, \theta, y_{1})-g(t, \theta, y_{2}) \| \leq k_1 {\|y_{1}- y_{2}\|},\ t>0,\ t-\tau\le \theta\le t,\  y_{1}, y_{2} \in\mathbb{R}^{n},
\end{equation}
where $a_1, b_1, c_1$, and $k_1$ are some given real constants.
Here the one-sided Lipschitz continuous condition is used instead of the classical Lipschitz continuous condition, 
mainly because it allows the problem to be stiff, especially for the semi-discrete system of parabolic equations. 

\item The continuous mapping
$f$ satisfies the dissipative structural condition
\begin{equation} \label{eq4}
\begin{split}
\left\langle u-Nv, f\left(t, u, v, w\right)\right\rangle \leq \gamma+ a_2\|u\|^2+b_2|v\|^2+c_2\|w\|^2,\;\;t>0, u, v, w \in\mathbb{R}^{n},
\end{split}
\end{equation}
and the continuous function 
$g$ satisfies that 
\begin{equation} \label{eq5}
\begin{split}
\|g(t, \theta, y)\| \leq k_2 {\|y\|},\;\;t>0, t-\tau\le \theta\le t,  y \in\mathbb{R}^{n},
\end{split}
\end{equation}
where $\gamma, a_2, b_2, c_2$, and $k_2$ are also some real constants.
The structural assumption of dissipativity usually results in an attractive set, that is, the solution always enters a bounded sphere after a finite time for any given initial values. 
Dissipativity is a typical characteristic of many nonlinear systems and is widely used in many practical models \cite{Wang, WangWS}. 
\end{itemize}

Before presenting the main result of this section, we recall below an important inequality.
\begin{lemma}{(\cite[Lemma 1, Remark 1]{CDG_14}, \cite[Theorem 2]{TuanIET18})}\label{Ine-Caputo}
		Let $x:[0,+\infty)\rightarrow \R^n$ be continuous and assume that the Caputo fractional derivative $^{\!C}D^{{\alpha}}_{0^+}x(\cdot)$ exists on $(0,\infty)$. Then, for any $t>0$, we have
	\begin{equation}\label{ineq_add}	
  ^{C}D^{{\alpha}}_{0^+}\langle x(t),x(t)\rangle \le 2 \langle x(t), ^{C}D^{{\alpha}}_{0^+}x(t)\rangle.
  \end{equation}
	\end{lemma}
	
 \begin{remark}
  Inequality \eqref{ineq_add} is a key tool in analyzing the asymptotic behavior of fractional differential equations. The original version was proposed by Aguila-Camacho, Duarte-Mermoud, and Gallegos \cite[Lemma 1, Remark 1]{CDG_14} for differentiable functions, and it was later extended by Trinh and Tuan \cite[Theorem 2]{TuanIET18} to Caputo fractionally differentiable functions.  
 \end{remark}

\begin{theorem}\label{thm:main}
(i) Let $y(t)$ be solution of (\ref{eq1}) with the stability conditions (\ref{eq2_ineq})--(\ref{eq3_ineq}). Assume that  $a_1<0$, $0\leq \left(2b_1-2a_1\|N\|^2+2c_1 k_1^{2}\tau^{2}\right)<-a_1$ and $2\|N\|^2 <1$. We then have the contractive stability inequality
\begin{equation} \label{eq18}
\begin{split}
\|y(t)-z(t)\|^2\leq  M_{1} E_{\alpha}(-\lambda^{*} t^{\alpha})~\text{ for }~t\geq 0,
\end{split}
\end{equation}
where $M_{1}=\max\limits_{-\tau\leq\xi\leq0}\|\varphi(\xi)-\chi(\xi)\|^2$, $z(t)$ denotes the solution for the perturbed problem
with the initial function $\chi(t)$, and the  parameter $\lambda^{*}>0$  is the unique solution of the following equation 
 \begin{equation}\label{eq19}
\begin{split}
h_{1}(\lambda):=\lambda+a_1+  2\left(b_1-a_1\|N\|^2+c_1 k_1^{2}\tau^{2}\right)/ E_{\alpha}(-\lambda\tau^{\alpha})=0. 
\end{split}
\end{equation}
Moreover, we have the Mittag-Leffler stablilty $\|y(t)-z(t)\|^2=O(t^{-\alpha})$ as $t\to \infty$.

(ii)  Let $y(t)$ be solution of (\ref{eq1}) with the dissipative conditions (\ref{eq4})--(\ref{eq5}). 
Assume that  $\gamma\geq 0$, $a_2<0$, $0\leq \left(2b_2-2a_2\|N\|^2+2c_2 k_2^{2}\tau^{2}\right)<-a_2$ and $2\|N\|^2 <1$. Then, for any $t\geq 0$, the following estimate hold
\begin{equation}\label{eq20}
\begin{split}
\|y(t)\|^{2} \le M_{2} E_{\alpha}(-\lambda^{*} t^{\alpha}) +\frac{-2\gamma}{a_1+2 \left(b_2-a_2\|N\|^2+c_2 k_2^{2}\tau^{2}\right)}, 
\end{split}
\end{equation}
where $M_{2}=\max\limits_{-\tau\leq\xi\leq0} \|\varphi(\xi)\|^2$ and the  parameter $\lambda^{*}>0$ is the unique solution of the following equation
 \begin{equation}\label{eq20_add}
h_{2}(\lambda):=\lambda+a_2+  2\left(b_2-a_2\|N\|^2+c_2 k_2^{2}\tau^{2}\right)/ E_{\alpha}(-\lambda\tau^{\alpha})=0.
\end{equation}
Further, the F-NFDEs in (\ref{eq1}) are dissipative, i.e., for any given $\varepsilon>0$, 
the ball $B\left(0,\sqrt{R}+\varepsilon\right)$ is an absorbing set, where $R:=\frac{-2\gamma}{a_1+2 \left(b_2-a_2\|N\|^2+c_2 k_2^{2}\tau^{2}\right)}$.
The Mittag-Leffler dissipativity rate is given by $\|y(t)\|^2\leq R+O(t^{-\alpha})$ as $t\to\infty$. 
\end{theorem}

\begin{proof}
(i) Let $e(t):=y(t)-z(t)$ for $t\geq -\tau$.
From \eqref{eq2_ineq} and \eqref{eq3_ineq}, we have
\begin{equation*}
\begin{split}
&\left\langle  e(t)-N e(t-\tau), ^{C}{D}_{0+}^{\alpha}(e(t)-N e(t-\tau)) \right\rangle\\
&=\left\langle  e(t)-N e(t-\tau),~ f(t, y(t), y(t-\tau), \int_{t-\tau}^{t} g(t, \xi, y(\xi)))-f(t, z(t), z(t-\tau), \int_{t-\tau}^{t} g(t, \xi, z(\xi)))\right\rangle\\
&\le a_1 \|e(t)\|^{2}+b_1 \|e(t-\tau)\|^{2}+c_1 \left\| \int_{t-\tau}^{t}g(t, \xi, y(\xi))d\xi -  \int_{t-\tau}^{t}g(t, \xi, z(\xi))d\xi\right \|^2\\
&\le a_1 \|e(t)\|^{2}+b_1 \|e(t-\tau)\|^{2}+c_1 k_1^{2} \tau^{2} \max_{t-\tau \leq \xi \leq t} \|e(\xi)\|^2.
\end{split}
\end{equation*}
It follows from Lemma \ref{Ine-Caputo} that 
\begin{equation}\label{eq21}
\begin{split}
^{C}{D}_{0+}^{\alpha} \left(\|e(t)-N e(t-\tau)\|^2\right) \le 2a_1 \|e(t)\|^{2}+2b_1 \|e(t-\tau)\|^{2}+2c_1 k_1^{2} \tau^{2} \max_{t-\tau \leq \xi \leq t} \|e(\xi)\|^2.
\end{split}
\end{equation}

Consider the auxiliary functions
$u(t)= \|e(t) -N e(t-\tau)\|^2 \text{ and } w(t)=\|e(t)\|^2.$
From the definitions, it is easy to see that $u(t)\le 2\|e(t)\|^2 +2\|N e(t-\tau)\|^2 \le 2w(t) +2\|N \|^2 w(t-\tau)$, which implies that 
$2a_1 w(t)\le a_1 u(t) - 2a_1 \|N \|^2 w(t-\tau)$ for the constant $a_1<0$. 
It follows from (\ref{eq21}) that
\begin{equation}\label{eq23}
\begin{split}
^{C}{D}_{0+}^{\alpha}u(t) &\leq 2a_1 w(t)+2b_1 w(t-\tau)+2c_1k_1^{2}\tau^{2} \max_{t-\tau \leq \xi \leq t} w(\xi)\\
 &\leq a_1 u(t)+\left(2b_1-2a_1\|N\|^2\right)w(t-\tau)+2c_1 k_1^{2}\tau^{2} \max_{t-\tau \leq \xi \leq t} w(\xi)\\
  &\leq a_1 u(t)+ \left(2b_1-2a_1\|N\|^2+2c_1 k_1^{2}\tau^{2}\right) \max_{t-\tau \leq \xi \leq t} w(\xi),\;\;t>0.\\
\end{split}
\end{equation}
 On the other hand, we have $\|e(t)\|=\|e(t)- N e(t-\tau)+N e(t-\tau)\|   \le \|e(t)- N e(t-\tau)\| +\|N e(t-\tau)\|$,
which yields that $w(t)\le 2u(t) + 2\|N \|^2 w(t-\tau)$ for $t\geq 0$.  We now get the coupled time fractional inequalities on $u(t)$ and $w(t)$ as
\begin{equation} \label{eq24}
\begin{split}
\left\{
\begin{aligned}
^{C}{D}_{0+}^{\alpha}u(t)  &\leq  a_1 u(t)+ \left(2b_1-2a_1\|N\|^2+2c_1 k_1^{2}\tau^{2}\right) \max_{t-\tau \leq \xi \leq t} w(\xi),\;\;t>0,\\
w(t)&\le 2u(t) +2 \|N \|^2 w(t-\tau),\;\;t\geq 0.
\end{aligned}
\right.
\end{split}
\end{equation}
The expected estimates are derived from the generalized fractional Halanay inequality in Theorem 2.4.

\noindent (ii) It follows from the dissipative condition (\ref{eq4}) that
\begin{equation*}
\begin{split}
&\left\langle y(t)-N y(t-\tau)\right), ^{C}{D}_{0+}^{\alpha}\left(y(t)-N y(t-\tau)\right\rangle\\
&=\left\langle y(t)-N y(t-\tau),  f(t, y(t), y(t-\tau), \int_{t-\tau}^{t}g(t, \xi, y(\xi))d\xi) \right\rangle\\
&\leq \gamma+ a_2\|y(t)\|^2+b_2\|y(t-\tau)\|^2+c_2 \left\| \int_{t-\tau}^{t}g(t, \xi, y(\xi))d\xi \right \|^2\\
&\leq \gamma+ a_2\|y(t)\|^2+b_2\|y(t-\tau)\|^2+c_2 k_2^{2}\tau^{2} \max_{t-\tau \leq \xi \leq t} \|y(\xi)\|^2.
\end{split}
\end{equation*}
 \noindent
By using Lemma \ref{Ine-Caputo}, then 
\begin{equation*}
\begin{split}
^{C}{D}_{0+}^{\alpha}\left(\|y(t) -N y(t-\tau)\|^2\right) \leq 2\gamma+ 2a_2\|y(t)\|^2+2b_2\|y(t-\tau)\|^2+2c_2 k_2^{2}\tau^{2} \max_{t-\tau \leq \xi \leq t} \|y(\xi)\|^2.
\end{split}
\end{equation*}
Similar to (i), we can consider auxiliary functions $u(t)= \|e(t) -N e(t-\tau)\|^2$ and $w(t)=\|y(t)\|^2$. After simple calculations, we can obtain inequalities on $u(t)$ and $w(t)$ as
\begin{equation} \label{eq26}
\begin{split}
\left\{
\begin{aligned}
^{C}{D}_{0+}^{\alpha}u(t)  &\leq 2\gamma+ a_2 u(t)+ \left(2b_2-2a_2\|N\|^2+2c_2 k_2^{2}\tau^{2}\right) \max_{t-\tau \leq \xi \leq t} w(\xi),\quad t>0,\\
w(t)&\le 2u(t) +2 \|N \|^2 w(t-\tau),\quad t\geq 0.
\end{aligned}
\right.
\end{split}
\end{equation}
Thus the desired results follow from the obtained fractional Halanay inequality. 
\end{proof}

\section*{Concluding remarks}
This paper presents a generalized fractional Halanay-type coupled inequality, which serves as both a natural extension and a strengthened version of several existing results in the literature. By combining this new inequality with the positive representation approach, we derive an asymptotic stability criterion for a class of fractional-order linear coupled systems with constant delays. To the best of our knowledge, this method offers a novel approach to studying the stability theory of fractional-order delay systems. As another application, we establish the dissipativity and contractility of time fractional neutral functional differential equations. For the key control parameter in the Mittag-Leffler function, its strict negativity can be more concisely proven.


\begin{thebibliography}{1}
\bibitem{CDG_14}
    N. Aguila-Camacho, M.A. Duarte-Mermoud, J.A. Gallegos, Lyapunov functions for fractional order systems. {\em Communications in Nonlinear Science and Numerical Simulation}, {\bf 19} (2014), no. 9, pp. 2951--2957.
\bibitem{Cong-Tuan-Hieu}
N.D. Cong, H.T. Tuan, and H.Trinh, {On asymptotic properties of solutions to fractional differential equations.} {\em Journal of Mathematical Analysis and Applications}, {\bf 484} (2020), 123759.
%
 \bibitem{Kai} K.~Diethelm,
\newblock{\em The Analysis of Fractional Differential Equations: An Application-Oriented Exposition Using Differential Operators of Caputo Type.} \newblock{ Lecture Notes in Mathematics,} {\bf 2004}.
\newblock{Springer-Verlag, Berlin, 2010.}

\bibitem{Met20}
R. Metzler, J. Klafter, The random walk’s guide to anomalous diffusion: a fractional dynamics approach, {\em Phys. Rep.,} {\bf 399} (2000), no. 1, pp. 1-77.

\bibitem{Gorenflo_B}
R. Gorenflo, A.A. Kilbas, F. Mainardi, and S.V. Rogosin, {\em Mittag-Leffler Functions, Related Topics and Applications.} {Springer Monogr. Math}.
 Springer, Heidelberg, 2014.
 
 \bibitem{HKT_23}
 N.T.T. Huong, N.N. Thang, and T.D. Ke, An improved fractional Halanay inequality with distributed delays. {\em Math. Meth. Appl. Sci.,} {\bf 46} (2023), pp. 19083--19099. 
 
\bibitem{NNThang}
N.T.T. Huong, N.N. Thang, and T.T.M. Nguyet, {Global fractional Halanay inequalities approach to finite-time stability of nonlinear fractional order delay systems.} {\em J. Math. Anal. Appl.,} {\bf 525} (2023), 127145.

\bibitem{Jin21}
B. Jin, Fractional Differential Equations-An Approach via Fractional Derivatives, Appl. Math. Sci. 206, Springer, Cham, Switzerland, 2021.

\bibitem{KT23}
T.D. Ke and N.N. Thang, An Optimal Halanay Inequality and Decay Rate of Solutions to Some Classes of Nonlocal Functional Differential Equations. {\em J Dyn Diff Equat.,} 2023. https://doi.org/10.1007/s10884-023-10323-w

\bibitem{Luliis}
 V.D. Iuliis, A. D'Innocenzo, A. Germani, and C. Manes, Internally positive representations and stability analysis of linear differential systems
with multiple time-varying delays. {\em IET Control Theory and Applications}, {\bf 13} (2019), no. 7, pp. 920--927.

\bibitem{TuanIET18}
				H.T. Tuan and H. Trinh, Stability of fractional-order nonlinear systems by Lyapunov direct method. {\em IET Control Theory
					Appl.,} {\bf 12} (2018), no. 17,  pp. 2417--2422.
					
\bibitem{Tuan-Thinh}
H.T. Tuan and L.V. Thinh, {\em Qualitative analysis of solutions to mixed-order positive linear coupled systems with bounded or unbounded delays.} {ESAIM - Control, Optimisation and Calculus of Variations,} {\bf 29} (2023), pp. 1--35.

\bibitem{Vainikko_16}
	G. Vainikko, Which functions are fractionally differentiable? 
	{\em Z. Anal. Anwend.}, {\bf 35} (2016), no. 4, pp. 465--487.

\bibitem{Zacher15}	
V. Vergara and R. Zacher, Optimal decay estimates for time-fractional and other non-Local subdiffusion equations via energy methods, SIAM J. Math. Anal., 47 (2015), pp. 210-239.

\bibitem{WHO}
H. O.Walther,  Topics in delay differential equations. Jahresbericht der Deutschen Mathematiker-Vereinigung, 2014, 116: 87-114.
	
\bibitem{Wang_15}
D. Wang, A. Xiao, and H. Liu, Dissipativity and Stability Analysis for Fractional Functional Differential Equations. {\em Fractional Calculus and applied Analysis,} {\bf 18} (2015), pp. 1399--1422. 
    
 \bibitem{Wang_19}   
 D. Wang and J. Zou, Dissipativity and Contractivity Analysis for Fractional Functional Differential Equations and their Numerical Approximations. {\em SIAM Journal on Numerical Analysis,} {\bf 57} (2019), no. 3, pp. 1445--1470.
 
  \bibitem{Wang_23}  
 D. Wang and J. Zou, Mittag-Leffler stability of numerical solutions to time fractional ODEs, {\em Numer. Algorithms}, {\bf 92} (2023), pp.  2125-2159.
 
\bibitem{Wang}
D. Wang, A. Xiao, and S. Sun, {Asymptotic behavior of solutions to time-fractional neutral functional differential equations.} {\em Journal of Computational and Applied Mathematics}, {\bf 382} (2021), 113086.

\bibitem{WangWS} 
W. Wang, Numerical analysis of nonlinear neutral functional differential equations (in Chinese), the second version, Science Press, Beijing, 2023.

\bibitem{ChengHu}
S. Yang, J. Yu, and C. Hu, {Finite-Time Synchronization of Memristive Neural Networks With Fractional-Order.} {\em IEEE Transactions on Systems, Man, and Cybernetics: Systems},  {\bf 51} (2021), no. 6, pp. 3739--3750.

\bibitem{Zhang}
				S. Zhang, M. Tang, X. Liu, and X.M. Zhang,
				Mittag–Leffler stability and stabilization of delayed fractional-order memristive neural networks based on a new Razumikhin-type theorem.
				{\em Journal of the Franklin Institute}, {\bf 361} (2024), no. 3, pp. 1211--1226.
\end{thebibliography}
\end{document}